\documentclass[11pt]{article}

\usepackage{hyperref}
\usepackage{graphicx}
\usepackage{mathrsfs}
\hypersetup{colorlinks, linkcolor=blue, urlcolor=red, filecolor=green, citecolor=blue}

\usepackage{amsmath}
\usepackage{amssymb}
\usepackage{latexsym}
\usepackage{amsthm}
\usepackage{fullpage}
\usepackage{graphicx}
\usepackage{color}

\usepackage[utf8]{inputenc}

\newcommand{\step}[1]{\medskip\noindent\textbf{Step #1. }}
\newcommand{\substep}[1]{\medskip\noindent\textit{Substep #1. }}

\graphicspath{{figures/}}

 {\begin{list}{}%
         {\setlength{\leftmargin}{#1}}%
         \item[]%
 }
 {\end{list}}

\newtheorem{theorem}{Theorem}[section]
\newtheorem{proposition}[theorem]{Proposition}
\newtheorem{lemma}[theorem]{Lemma}

\newtheorem{definition}{Definition}[section]

\newtheorem{assumption}{Assumption}[section]

\author{Laura Lauerbach\footnote{Institute of Mathematics, University of Kassel, Heinrich-Plett-Straße 40,
34132 Kassel, Germany},
Stefan Neukamm\footnote{Faculty of Mathematics, Technische Universit\"{a}t Dresden, 01069 Dresden, Germany},
Mathias Sch\"{a}ffner\footnote{Faculty of Mathematics, TU Dortmund, 44227 Dortmund, Germany},
Anja Schl\"{o}merkemper\footnote{Institute of Mathematics,
University of W\"{u}rzburg, Emil-Fischer-Str.~40, 97074 W\"{u}rzburg, Germany}
}

\title{Onset of fracture in random heterogeneous particle chains}

\begin{document}
\maketitle

\begin{abstract}
In mechanical systems it is of interest to know the onset of fracture in dependence of the boundary conditions. Here we study a one-dimensional model which allows for an underlying heterogeneous structure in the discrete setting. Such models have recently been studied in the passage to the continuum by means of variational convergence ($\Gamma$-convergence). The $\Gamma$-limit results determine thresholds of the boundary condition, which mark a transition from purely elastic behaviour to the occurrence of a crack. In this article we provide a notion of fracture in the discrete setting and show that its continuum limit yields the same threshold as that obtained from the $\Gamma$-limit. Since the calculation of the fracture threshold is much easier with the new method, we see a good chance that this new approach will turn out useful in applications.
\end{abstract}

\medskip



\section{Introduction}
The mechanical behaviour of one-dimensional systems has been of interest for decades. Such systems serve as toy models for higher-dimensional theoretical investigations and are of interest with respect to one-dimensional structures, see, e.g., \cite{FriedrichStefanelli2020,HallHudsonvanMeurs2018,JansenKoenigSchmidtTheil2021,KimuravanMeurs2021, Truskinovsky1996}. In order to understand the effective behaviour of materials, the systems are studied as the number of particles tends to infinity.

\medskip

In this article we focus on the occurrence of cracks and continue a mathematical analysis of the effective behaviour of one-dimensional discrete systems in the passage to the continuum. In particular we strive for insight into the threshold for the overall prescribed length $\ell$ of a chain. If $\ell$ is smaller than the threshold, the system will show elastic behaviour, whereas cracks are energetically favoured if $\ell$ is larger than the threshold. The interaction potentials between the particles or atoms of the discrete chain are allowed to be in a large class of convex-concave potentials, which include for instance the classical Lennard-Jones potentials. The system is then modeled with the help of an energy functional that is the sum of all the interaction potentials, see \eqref{eq:energy=pre}. Here we restrict to the interactions of nearest neighbours; for related studies with interactions beyond nearest neighbours we refer to \cite{BraidesLewOrtiz2006,BraidesSolci2016,SchaeffnerSchloemerkemper2015}.

\medskip
In view of misplaced atoms or of chains consisting of several different kind of particles, we allow for a random distribution of the interaction potentials, see  Assumption~\ref{Ass:stochasticLJ} and \eqref{eq:energy} for details. The limit passage is then also referred to as stochastic homogenization, cf., e.g., \cite{AlicandroCicaleseGloria2011, DalMasoModica1985,IosifescuLichtMichaille2001,NeukammSchaeffnerSchloemerkemper2017}. As a special case, also materials with a periodic heterostructure are included, cf.\ also \cite{LauerbachSchaeffnerSchloemerkemper2017}.

\medskip

An appropriate mathematical technique for the passage of energy functionals from  discrete to continuous systems is based on the notion of $\Gamma$-convergence, which is a notion of a variational convergence and (under coercivity assumptions) ensures that minimizers of the discrete system converge to minimizers of the system in the continuum limit, see, e.g., \cite{BraidesDalMasoGarroni1999,BraidesGelli2002,ScardiaSchloemerkemperZanini2011} 
and references cited therein. As the number of particles tends to infinity, the energy functional converges to a functional that shows the occurrence of cracks in dependence of the boundary conditions that determine the length of the one-dimensional structure. We stress that, on the discrete level, no crack has been identified so far. Instead, the mathematical analysis works with a piecewise affine and thus continuous deformation of the discrete system. Still, a deeper understanding of the crack formation is also of interest on the discrete level.

\medskip

In this article, which is partially based on the PhD thesis \cite[Chapter~7]{Lauerbach-Diss} of L.~Lauerbach,  we focus on the emergence of cracks in atomistic chains. On the level of the continuum limiting model of the chain, ``crack'' has a clear meaning -- it is the point where the continuum deformation features  a jump and there is no interaction between the different segments separated by the jump. In contrary, on the level of a discrete chain with $n+1$ particles, the notion of ``crack'' cannot be unambigously defined, since always neighbouring particles interact. In the present paper we introduce a notion of ``onset of a crack'' at the discrete level for a chain with $n+1$ particles. For simplicity we discuss the key idea in the case of a chain with $n+1$ particles that is composed of (random) potentials that are convex around its ground state and otherwise concave, i.~e.~for deformations larger than an inflection point $z_{\mathrm{frac}}$. We call a deformation $u$ \textit{elastic}, if the individual interaction potentials along the chain are only evaluated in their convex region. In contrary, a deformation that is \textit{not} elastic invokes at least one bond that ``lives'' in the concave region of the corresponding potential. Next, we consider the energy minimizers $u_n$ of the chain with $n+1$ particles and prescribed total length $\ell>0$. If the minimizers $u_n$ are elastic for all $n\in\mathbb{N}$, then we do not expect the occurance of crack in the continuum limit; while in the other case, we expect that minimizing sequeces show a concentration of strain on a finite number of weak bonds and thus a ``crack'' emerges in the continuum limit. Based on this heuristics we intrduce a ``critical stretch'' $\ell^*_n$ for random chains with $n+1$ particles. Firstly, we prove that it converges, for $n\to\infty$, to the jump-threshold predicted by the zeroth-order $\Gamma$-limit of the discrete energy, which has been obtained earlier in \cite{unserpaper1}. Secondly, we establish a first order expansion of the critical stretch and show that the coefficients of the expansion term agree with the values predicted by the first-order $\Gamma$-limit of the discrete energy derived in \cite{unserpaper1}. Since the proofs in \cite{unserpaper1} are technically quite involved, it is interesting to learn that there is a much simpler method for the derivation of the jump threshold in the continuum limit. We expect that the new notions of a fracture point and of a jump threshold in the discrete setting turn out to be useful also in a wider class of applications.

\medskip
The outline of this article is as follows: In the second section we introduce the model in the discrete setting, including the assumptions on the large class of interaction potentials in the random setting. Further we provide the definition of a critical stretch (Definition~\ref{def:final}), which corresponds to the jump threshold. We assert the asymptotic behaviour of the critical stretch as the number of particles tends to infinity (Theorem~\ref{thm:deltaverschieden}) and compare the limit to the corresponding $\Gamma$-convergence results. Moreover, we consider a rescaled setting, define the rescaled jump threshold and assert its asymptotic behaviour as $n$ tends to infinity (Theorem~\ref{thm:deltagleich}). Finally we compare also this result with the corresponding $\Gamma$-convergences result.
All proofs are provided in Section~\ref{sec:proofs}.

\section{Setup and main results}

\begin{figure}[t]
\centering
		\includegraphics[width=0.6\linewidth]{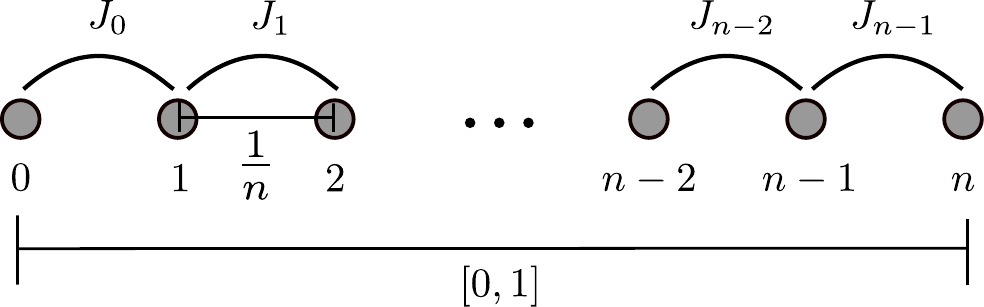}
		\caption{Chain of $n+1$ atoms with reference position $\frac{i}{n}$. The potential $J_i$ describes the nearest neighbour interaction of atom $i$ and $i+1$, $i=0,...,n-1$. The characteristic length scale is $\frac{1}{n}$ and the interval is $[0,1]$.}
\label{fig:kette}
\end{figure}

We consider a chain of $n+1$ atoms that in a reference configuration are placed at the sites in $\frac{1}{n}\mathbb{Z}\cap[0,1]$, see Figure~\ref{fig:kette}. The deformation of the atoms is referred to as $u_n: \frac{1}{n}\mathbb{Z}\cap[0,1]\rightarrow\mathbb{R}$. For the passage from discrete systems to their continuous counterparts it is useful to identify the discrete functions with their piecewise affine interpolations, i.e., with the functions in
\begin{align*}
	\mathcal{A}_n:=\left\lbrace u \in C([0,1]): u\ \text{is affine on}\ (i,i+1)\frac{1}{n},\ i\in\{0,1,...,n-1\}, \text{monotonically increasing} \right\rbrace.
\end{align*}
We shall also consider clamped boundary conditions for the chain and thus introduce for $\ell>0$ the set
\begin{equation*}
  \mathcal A_{n,\ell}:=\{u\in\mathcal A_n\,:\,u(0)=0,\,u(1)=\ell\}.
\end{equation*}
We consider a discrete energy functional of the form
\begin{align} \label{eq:energy=pre}
  \mathcal A_{n,\ell}\ni u\mapsto E_n(u):=\sum_{i=0}^{n-1}\frac{1}{n}J_i\left(\dfrac{u\left(\frac{i+1}{n}\right)-u\left(\frac{i}{n}\right)}{\frac{1}{n}}\right)=\sum_{i=0}^{n-1}\frac{1}{n}J_i\left(n\left(u\left(\frac{i+1}{n}\right)-u\left(\frac{i}{n}\right)\right)\right),
\end{align}
where $J_i:(0,\infty)\to\mathbb R$ is a potential describing the interaction between the $i$th atom and its neighbour to the right. We are interested in random heterogeneous chains of atoms, and thus assume that the potentials $\{J_i\}_{i\in\mathbb Z}$ are random with a distribution that is stationary and ergodic. We appeal to the following standard setup: Let $(\Omega,\mathcal F,\mathbb P)$ denote a probability space and $(\tau_i)_{i\in\mathbb Z}$ a family of measurable maps $\tau_i:\Omega\rightarrow\Omega$ such that
\begin{itemize}
	\item (Group property.) $\tau_0\omega=\omega$ for all $\omega\in\Omega$ and $\tau_{i_1+i_2}=\tau_{i_1}\tau_{i_2}$ for all $i_1,i_2\in\mathbb{Z}$,
	\item (Stationarity.) $\mathbb{P}(\tau_iB)=\mathbb{P}(B)$ for every $B\in\mathcal{F}$ $i\in\mathbb{Z}$,
	\item (Ergodicity.) For all $B\in\mathcal{F}$, it holds $(\tau_i(B)=B\ \forall i\in\mathbb{Z})\Rightarrow\mathbb{P}(B)=0\ \text{or}\ \mathbb{P}(B)=1$.
\end{itemize}
We then consider the energy functional
\begin{align}\label{eq:energy}
  E_n:\Omega\times\mathcal A_n\to\mathbb R\cup\{+\infty\},\qquad E_n(\omega,u):=\sum_{i=0}^{n-1}\frac{1}{n}J\left(\tau_i\omega,n\left(u\left(\frac{i+1}{n}\right)-u\left(\frac{i}{n}\right)\right)\right),
\end{align}
where the random potential satisfies the following assumptions:
\begin{assumption}
	\label{Ass:stochasticLJ}
  Let $J:\Omega\times\mathbb R\to\mathbb R\cup\{+\infty\}$ be jointly measurable  with $J(z)=\infty$ if $z\leq 0$. For $\mathbb P$-a.e.~$\omega\in\Omega$ the following conditions hold true:
  \begin{enumerate}
  \item[(A1)] (Regularity.) $J(\omega,\cdot)\in C^3(0,\infty)$. 
  \item[(A2)]\label{ass:A2} (Behavior at $0$ and $\infty$.) There exists functions  $\psi^+,\psi^-\in C(0,\infty)$, independent of $\omega$, such that
    \begin{equation*}
      \lim\limits_{z\to 0^+}\psi^-(z)=\infty\text{ and }      \lim\limits_{z\to \infty}\psi^+(z)=0,
    \end{equation*}
    and
    \begin{equation*}
      J(\omega,z)\geq \psi^-(z)\text{ for all }0<z\leq1\quad\text{and}\quad |J(\omega,z)|\leq\psi^+(z)\text{ for all }z\geq 1.
    \end{equation*}
  \item[(A3)] (Convex-monotone-structure.) Suppose strict convexity close to $0$ in form of
    \begin{equation*}
        z_{\mathrm{frac}}(\omega):=\sup\Big\{z>0\,:\,J''(\omega,s):= \partial^2_s J(\omega,s)> 0\text{ for all }s\in (0,z)\,\Big\}>0,
      \end{equation*}
      and assume that $J(\omega,\cdot)$ is monotonically increasing on $[z_{\mathrm{frac}}(\omega),\infty)$.
    \end{enumerate}
    By the previous assumptions $J(\omega,\cdot)$ is strictly convex in $(0,z_{\mathrm{frac}}]$ and admits a unique minimizer $\delta(\omega)$ that we call the ground state of $J(\omega,\cdot)$. We assume its non-degeneracy in the following sense:
    \begin{enumerate}
    \item[(A4)] (Non-degenerate ground state.) Suppose that there exists a constant $c>0$, independent of $\omega$, such that $\frac1c> \delta(\omega)>c$ and 
      \begin{equation*}
        \forall z\in \delta(\omega)+(-c,c)\,:\,c\leq J''(\omega,z)\leq\frac1c\text{ and }
        |J'''(\omega,z)|\leq \frac{1}{c}.
      \end{equation*}
  \end{enumerate}
\end{assumption}

Next, we introduce the following central quantities for a random heterogeneous chain with $n+1$ particles:

\begin{definition}[Critical stretch of a chain with $n+1$ particles] \label{def:final}
  Consider the situation of Assumption~\ref{Ass:stochasticLJ}. Let $n\in\mathbb N$ and $\omega\in\Omega$. The critical stretch $\ell^*_n(\omega)$ is defined as the largest number  such that
    \begin{equation*}
      \inf_{\mathcal A^{\rm el}_n(\omega)\cap\mathcal A_{n,\ell}}E_n(\omega,\cdot)=\inf_{\mathcal A_{n,\ell}}E_n(\omega,\cdot)\qquad\text{for all }0\leq\ell<\ell^*_n(\omega),
    \end{equation*}
  where we denote by
  \begin{equation*}
      \mathcal A^{\rm el}_n(\omega):=\Bigg\{u\in\mathcal A_n\,:\,\frac{u^{i+1}-u^i}{\frac{1}{n}}\leq       z_{\mathrm{frac}}(\tau_i\omega)\text{ for all }i=0,\ldots,n\Bigg\}
    \end{equation*}
    the set of purely elastic deformations.
\end{definition}
The idea behind the above definition is the following: A deformation $u\in\mathcal A^{\rm el}_n(\omega)$ only sees the strictly convex region of the interaction potentials. Thus, we could replace the potentials $J(\omega,i,z)$ in the definition of the energy function $E_n$ by (globally) convex potentials with superlinear growth without changing the energy for deformations in $\mathcal A_n^{\rm el}(\omega)$. As it is well-known, such energies do not allow for fracture in the continuum limit. The definition of the critical stretch implies that a prescribed macroscopic stretch (or compression) $\ell<\ell^*_n(\omega)$ can be realized by a deformation in $\mathcal A_n^{\rm el}(\omega)$ and thus prohibits the formation of a jump, while, for $\ell>\ell^*_n(\omega)$, deformations with minimial energy are required to explore the non-convex region of at least one of the interaction potentials. We may refer to the bonds $[i,i+1]$ that are evaluated outside the convex region as ``weak'' bonds. If a jump occurs in the limit, then the minimizing sequence shows a concentration of strain in the weak bonds.
We thus expect that $\ell^*_n(\omega)$ almost surely converges in the limit $n\to\infty$ to the continuum fracture threshold that can be defined on the level of the continuum $\Gamma$-limit, see below. In our first result we prove that  $\ell^*_n$ indeed converges and we identify its limit, which is the statistical mean of the ground states: 
\begin{theorem}
	\label{thm:deltaverschieden}
	Let Assumption~\ref{Ass:stochasticLJ} be fulfilled. Then,
	\begin{align*}
	\lim\limits_{n\to\infty}\ell_n^*(\omega)=\mathbb{E}[\delta]\quad\text{for $\mathbb{P}$-a.e. }\omega\in\Omega.
	\end{align*}
	
\end{theorem}
(The proof of Theorem~\ref{thm:deltaverschieden} can be found in Section~\ref{sec:proofthmdeltaverschieden}.)\\

Next we consider the special case when $\delta(\omega)$ is deterministic, say $\delta(\omega)=1$ for $\mathbb P$-a.s. In that case we establish a first order expansion of $\ell^*_n(\omega)$ around its limit $\mathbb E[\delta]=1$ of the form
\begin{equation*}
  \ell_n^*(\omega)\approx 1+\sqrt{\frac{1}{n}}\sqrt{\frac{\beta}{\underline\alpha}},
\end{equation*}
where $\beta$ is related to the maximal energy barrier among the random potentials $J$, and $1/\underline\alpha$ is the statistical mean of the curvatures of the random potentials at the ground state.
\begin{theorem}
	\label{thm:deltagleich}
	Let Assumption~\ref{Ass:stochasticLJ} be satisfied and assume $\delta(\omega)=1$ for $\mathbb P$-a.~e.~$\omega\in\Omega$. Consider the rescaled jump threshold $\gamma_n^*(\omega):=\dfrac{\ell_n^*(\omega)-1}{\sqrt{\frac{1}{n}}}$.
        Then
	\begin{align*}
          \lim\limits_{n\to\infty}\gamma_n^*(\omega)=\lim\limits_{n\to\infty}\dfrac{\ell_n^*(\omega)-1}{\sqrt{\frac{1}{n}}}=\sqrt{\frac{\beta}{\underline{\alpha}}}\qquad\text{for }\mathbb P\text{-a.e.~}\omega\in\Omega,
	\end{align*}
	where
        \begin{equation} \label{def:alphabeta}
          \underline\alpha:=\left(\mathbb{E}\Big[\big(\tfrac{1}{2} J''(\omega,1)\big)^{-1}\Big]\right)^{-1}\text{ and }\quad \beta:={\rm ess}\inf_{\omega\in\Omega}(-J(\omega,1)).
          \end{equation}
\end{theorem}
(The proof of Theorem~\ref{thm:deltagleich} can be found in Section~\ref{sec:proofthmdeltagleich}.)\\

We finally relate the above results to the zeroth- and first-order $\Gamma$-limits of $E_n$ subject to clamped boundary conditions, i.e.~
\begin{equation*}
  E_n^\ell(\omega,\cdot):L^1(0,1)\to\mathbb R\cup\{+\infty\},\quad E_n^\ell(\omega,u):=
  \begin{cases}
    E_n(\omega,u)&\text{if }u\in\mathcal A_{n,\ell},\\
    +\infty&\text{else.}
  \end{cases}
\end{equation*}
The zeroth-order $\Gamma$-limit of the discrete energy yields a homogenized energy functional
\begin{align*}
  E_{\mathrm{hom}}^{\ell}=\int_{0}^{1}J_{\mathrm{hom}}(u'(x)),
\end{align*}
recalled from \cite{unserpaper1}, and adjusted to the stronger assumptions of this paper and $K=1$, 
	where the homogenized energy density map $z\mapsto J_{\mathrm{hom}}(z)$ is convex, lower semicontinuous, monotonically decreasing and satisfies
	\begin{align}\label{superlinwachstum}
	\lim\limits_{z\rightarrow0^+}J_{\mathrm{hom}}(z)=+\infty.
	\end{align}
	Moreover, the minimum values of $E_{n}^{\ell}(\omega,\cdot)$ and $E_{\mathrm{hom}}^\ell$ satisfy
	\begin{align*}
	\lim\limits_{n\to\infty}\inf_u E_{n}^{\ell}(\omega,u)=\min_u E_{\mathrm{hom}}^{\ell}(u)=J_{\mathrm{hom}}(\ell),
	\end{align*}
    and therefore can be calculated as
\begin{align*}
\min_u E_{\mathrm{hom}}^{\ell}(u)=J_{\mathrm{hom}}(\ell)=\begin{cases}
J_{\mathrm{hom}}(\ell)&\quad\text{for}\ \ell<\mathbb{E}[\delta],\\[1mm]
J_{\mathrm{hom}}(\mathbb{E}[\delta])&\quad\text{for}\ \ell\geq\mathbb{E}[\delta].
\end{cases}
\end{align*}
Hence, the threshold between the elastic and the jump regimes is $\mathbb{E}[\delta]$, which is identical to the limit of $\ell^*_n(\omega)$, see Theorem~\ref{thm:deltaverschieden}. Secondly,
we recall a $\Gamma$-limit result from \cite{unserpaper2} for the rescaled energy functional
\begin{align*}
  H_n(\omega,v):=\sum_{i=0}^{n-1}\left(J\left(\tau_i\omega,\dfrac{v\left(\frac{i+1}{n}\right)-v\left(\frac{i}{n}\right)}{\sqrt{\frac{1}{n}}}+\delta(\tau_i\omega)\right)-J\left(\tau_i\omega,\delta(\tau_i\omega)\right) \right),
\end{align*}
where the $\Gamma$-limit is given by
\begin{align*}
H^{\gamma}(v)=\underline{\alpha}\displaystyle \int_{0}^{1}\left|v'(x) \right|^2\,\mathrm{d}x+\beta\#S_v,
\end{align*}
with homogenized elastic coefficient $\underline{\alpha}$, jump parameter $\beta$ and $\#S_v$ being the number of jumps of $v$.
	
	Moreover, for $\gamma>0$ it holds true that
	\begin{align*}
	\lim\limits_{n\rightarrow\infty}\inf_v H_n^{\gamma_n}(\omega,v)=\min_vH^{\gamma}(v)=\min\{\underline{\alpha}\gamma^2,\beta\},
	\end{align*}
which yields that the minima of the energy are given by
\begin{align*}
\min_vH^{\gamma}(v)=\min\{\underline{\alpha}\gamma^2,\beta\}=\begin{cases}
\underline{\alpha}\gamma^2&\quad\text{if}\ \gamma<\sqrt{\frac{\beta}{\underline{\alpha}}},\\[2mm]
\beta&\quad\text{if}\ \gamma\geq\sqrt{\frac{\beta}{\underline{\alpha}}}.
\end{cases}
\end{align*}
Hence the threshold between elasticity and fracture in the rescaled case is $\sqrt{\frac{\beta}{\underline{\alpha}}}$, which equals the limit of the jump threshold $\gamma_n^*$ in Theorem~\ref{thm:deltagleich}.

In summary, although the techniques by which the results are calculated are completely different, they yield the same result regarding the jump threshold in the continuum setting. The derivation of the limiting jump threshold with help of the newly defined jump threshold in the discrete setting is, however, much easier and thus is of interest for applications. It remains an open problem to analyze corresponding questions in higher dimensional settings. In the following section we provide the proofs of the above theorems.

\section{Proofs} \label{sec:proofs}

For the upcoming analysis it is convenient to introduce the notation 
  \begin{equation*}
    M_n(\omega,\ell):=\min\biggl\{\frac1n\sum_{i=1}^{n}J(\tau_i\omega,z^i)\,: \,\frac1n\sum_{i=1}^nz^i=\ell\,\biggr\}
  \end{equation*}
  to denote the minimum energy of a discrete chain of length $\ell$.  We begin with an elementary (yet, convenient) reformulation of the critical stretch $\ell_n^*$ (cf.~Definition~\ref{def:final}).

\begin{lemma}\label{L:reformellstar}
  Consider the situation of Assumption~\ref{Ass:stochasticLJ}. Let $n\in\mathbb N$ and $\omega\in\Omega$. Then, it holds
  \begin{align}\label{def:Mnell}
    M_n(\omega,\ell)=\min_{u\in\mathcal A_{n,\ell}}E_n(\omega,u).
  \end{align}
  Moreover, $\ell_n^*(\omega)$ is the largest number such that for all $0<\ell<\ell_n^*(\omega)$ there exists $\bar z\in\mathbb R^n$ satisfying
  \begin{equation}\label{eq:lemma1}
    M_n(\omega,\ell)=\frac1n\sum_{i=1}^nJ(\tau_i\omega,\bar z^i),\quad \frac1n\sum_{i=1}^n\bar z^{i}=\ell,\quad \bar z^i\leq z_{\rm frac}(\tau_i\omega)\quad\forall i\in\{1,\dots,n\}.
  \end{equation}
\end{lemma}

\begin{proof}[Proof of Lemma~\ref{L:reformellstar}]

The identity \eqref{def:Mnell} follows by a simple change of variables, that is by setting $z^i=n(u(\frac{i}n)-u(\frac{i-1}n))$, and the direct method of the calculus of variations. 

Next we give an argument regarding the characterization of $\ell_n^*$. The definition of $\ell_n^*(\omega)$, see  Definition~\ref{def:final}, and \eqref{def:Mnell} imply
$$
\inf_{\mathcal A^{\rm el}_n(\omega)\cap\mathcal A_{n,\ell}}E_n(\omega,\cdot)=M_n(\omega,\ell)<\infty\qquad\forall \ell\in (0,\ell_n^*(\omega)).
$$
Since $\mathcal A^{\rm el}_n(\omega)\cap\mathcal A_{n,\ell}$ is compact, there exists $\bar u\in \mathcal A^{\rm el}_n(\omega)\cap\mathcal A_{n,\ell}$ such that 
$$E_n(\omega,\bar u)=\inf_{\mathcal A^{\rm el}_n(\omega)\cap\mathcal A_{n,\ell}}E_n(\omega,\cdot).$$
Clearly, $\bar z\in \mathbb R^n$ defined as $\bar z^{i}=n(\bar u(\frac{i}n)-\bar u(\frac{i-1}i))$ satisfies \eqref{eq:lemma1}. 

Now we suppose that for some $\ell\geq \ell_n^*$ there exists $\bar z\in\mathbb R^n$ satisfying \eqref{eq:lemma1}. With help of the same change of variables as above, we find $\bar u\in\mathcal A^{\rm el}_n(\omega)\cap\mathcal A_{n,\ell}$ satisying $E_n(\omega,\bar u)=M_n(\omega,\ell)$ which contradicts the definition of $\ell_n^*$.
\end{proof}

\begin{lemma}\label{prop:derivatives}
Let Assumption~\ref{Ass:stochasticLJ} be satisfied. Then, $J(\omega,\cdot)$ is increasing on $[\delta(\omega),\infty)$ and it holds
\begin{equation}\label{def:zfracsup}
z_{\mathrm{frac}}^{\mathrm{sup}}:=\sup\{ z_{\mathrm{frac}}(\omega)\, :\, \omega\in\Omega\}<\infty.
\end{equation}
\end{lemma}

\begin{proof}[Proof of Lemma~\ref{prop:derivatives}]
  For convenience we drop the dependence on $\omega$ in our notation and simply write $J(z)$, $\delta$ and $z_{\mathrm{frac}}$ instead of $J(\omega,z)$, $\delta(\omega)$ and $z_{\mathrm{frac}}(\omega)$, respectively. We first prove that $J$ is increasing on $[\delta,\infty)$. On $[z_{\mathrm{frac}},\infty)$ this directly follows from (A3). On $[\delta,z_{\mathrm{frac}})$ this follows from the convexity of $J$ on $(0,z_{\rm frac})$ and the fact that $\delta$ minimizes $J$.
  Next, we prove  \eqref{def:zfracsup}. We first note that (A2) and (A3) imply that
  \begin{equation}\label{neu:1}
    \forall z\in(\delta,\infty)\,:\,J(\delta)\leq J(z)\leq 0.
  \end{equation}
  Moreover, (A4) implies that $z_{\rm frac}\geq\delta+c$. Thus, for all $\eta\in(0,c)$ we obtain
  \begin{align}\label{ineq:prop:derivatives}
    0\geq& J(z_{\rm frac})=J(\delta+\eta)+\int_{\delta+\eta}^{z_{\rm frac}}J'(t)\,\mathrm{d}t\geq J(\delta+\eta)+J'(\delta+\eta)(z_{\rm frac}-(\delta+\eta)), 
\end{align}
where the second inequality holds, since $J'$ is increasing on $(\delta+\eta, z_{\rm frac})$ thanks to (A3). (A4) yields
$$
J'(\delta+\eta)=J'(\delta+\eta)-J'(\delta)=\int_\delta^{\delta+\eta}J''(s)\,ds\geq c\eta.
$$
Thus, by rearranging terms in \eqref{ineq:prop:derivatives} and appealing to \eqref{neu:1} and the previous estimate we get
\begin{equation}\label{neu:2}
  z_{\rm frac}\leq \delta+\eta -\frac{J(\delta+\eta)}{J'(\delta+\eta)}\leq \delta+\eta -\frac{J(\delta)}{c\eta}.
\end{equation}
It remains to bound $\delta=\delta(\omega)$ and $-J(\delta)=-J(\omega,\delta(\omega))$ by a constant that is independent of  $\omega$. From (A4) and (A2) we get
\begin{equation}\label{neu:3}
\delta\in(c,\frac1c)\text{ and }
-J(\delta)\leq \max_{z\in [c,\frac1c+\eta]}\max\{-\psi^-(z),\,|\psi^+(z)|\}=:d<\infty,
\end{equation}
and thus, \eqref{neu:2} yields $z_{\rm frac}\leq \frac1c+\frac{d}{c\eta}$.
\end{proof}

\subsection{Proof of Theorem~\ref{thm:deltaverschieden}}
\label{sec:proofthmdeltaverschieden}

\begin{proof}[Proof of Theorem~\ref{thm:deltaverschieden}]
 Note that $\omega\mapsto\delta(\omega)$ is (as a minimizer of a measurable function) measurable. Moreover, by \eqref{neu:3}  $\delta$ is a non-negative and bounded and thus an $L^1$-random variable.
  Thus the ergodic theorem yields
\begin{equation}\label{lim:ergodic:T1}
\lim_{n\to\infty}\frac1n\sum_{i=1}^n\delta(\tau_i\omega)=\mathbb E[\delta],\quad \lim_{n\to\infty}\frac1n\sum_{i=1}^nJ(\tau_i\omega,\delta(\tau_i\omega))=\mathbb E[J(\delta)]\quad\mbox{}
\end{equation}	
for $\mathbb P$-a.e.\ $\omega\in\Omega$. For the rest of the proof we consider $\omega\in\Omega$ such that \eqref{lim:ergodic:T1} is valid and drop the dependence on $\omega$. In particular, we set $\delta_i:=\delta(\tau_i\omega)$, $z_{\mathrm{frac}}^i:=z_{\mathrm{frac}}(\tau_i\omega)$ and $J_i(z):=J(\tau_i\omega,z)$.

\step 1 We show that $\overline A:=\limsup_{n\to\infty}\ell_n^*\leq \mathbb{E}[\delta]$.

W.l.o.g.\ we suppose $\overline A=\lim_{n\to\infty}\ell_n^*$ and prove $\overline A\leq \mathbb E[\delta]$ by contradiction. Assume that there exists $\epsilon\in(0,c)$ such that $\overline A>\mathbb{E}[\delta]+3\epsilon$. By \eqref{lim:ergodic:T1}, we find $\overline N\in\mathbb N$ such that
	\begin{align}
	\label{abschaetzung}
	 \ell_n^*>\frac1n\sum_{i=1}^{n}\delta_i+2\epsilon=:k_n\quad\text{for}\ n>\overline{N}.
	\end{align}
In view of Lemma~\ref{L:reformellstar} there exists a sequence $(\bar z_n)_n$  satisfying for $n\geq \overline N$
\begin{equation}\label{eq:T1bs1zn}
 \frac1n\sum_{i=1}^n\bar z_n^{i}=k_n,\quad \frac1n\sum_{i=1}^nJ_i(\bar z_n^i)=M_n(k_n),\quad \bar z_n^i\leq z_{\rm frac}^i\quad\forall i\in\{1,\dots,n\}.
\end{equation}
We claim that
\begin{align}
\limsup_{n\to\infty}M_n(k_n)\leq& \mathbb{E}[J(\delta)], \label{pf:T1:Step1a}\\
\liminf_{n\to\infty}M_n(k_n)\geq&\mathbb{E}[J(\delta)]+c_\epsilon,\label{pf:T1:Step1b}
\end{align}
for some $c_\epsilon>0$. Clearly, \eqref{pf:T1:Step1a} and \eqref{pf:T1:Step1b} contradict \eqref{eq:T1bs1zn}

\substep{1.1} Proof of \eqref{pf:T1:Step1a}. Let $z_n\in\mathbb R^n$ be given by $z_n^i:=\delta_i$ for $i\geq2$ and $z_n^1:=\delta_1+2n\epsilon$. Since $\frac1n\sum_{i=1}^{n}z_n^i=k_n$, we have 
	\begin{align*}
	M_n(k_n)&\leq\frac1n\sum_{i=2}^{n}J_i(\delta_i)+\frac1nJ_1\left(\delta_1+2n\epsilon \right)=\frac1n\sum_{i=1}^{n}J_i(\delta_i)+\frac1n(J_1\left(\delta_1+2n\epsilon \right)-J_1(\delta_1)).
	\end{align*}
	Hence, \eqref{pf:T1:Step1a} follows by (A2) and \eqref{lim:ergodic:T1}.
	
	\substep{1.2} Proof of \eqref{pf:T1:Step1b}. Let $\bar z_n$ be as in \eqref{eq:T1bs1zn} and set
	\begin{align*}
	I_n:=\left\{i\in\{1,...,n\}:\bar z_n^i>\delta_i+\epsilon\right\}.
	\end{align*}
	Obviously, it holds $0\leq |I_n|/n\leq 1$ and we claim 
	\begin{equation}\label{claim:Lambda}
	\frac{|I_n|}n\geq \frac{\epsilon}{z_{\rm frac}^{\mathrm{sup}}}>0\qquad\mbox{for all $n\in\mathbb N$},
	\end{equation}
	where $z_{\mathrm{frac}}^{\mathrm{sup}}\in(0,\infty)$ is as in Lemma~\ref{prop:derivatives}. Indeed, 
	\begin{align*}
	\frac1n\sum_{i=1}^n\delta_i+2\epsilon=k_n=\dfrac{1}{n}\sum_{i=1}^{n}\bar z_n^i=\dfrac{1}{n}\sum_{i\in I_n}\bar z_n^i+\dfrac{1}{n}\sum_{i\notin I_n}\bar z_n^i\stackrel{\eqref{eq:T1bs1zn}}{\leq} \dfrac{|I_n|}{n}z_{\mathrm{frac}}^{{\rm sup}}+\dfrac{1}{n}\sum_{i=1}^n(\delta_i+\epsilon) 
	\end{align*}
implies \eqref{claim:Lambda}. Finally, using the monotonicity of $J_i$ on $(\delta_i,\infty)$ (see Lemma~\ref{prop:derivatives}) and (A4), we obtain
\begin{align*}
 \frac1n\sum_{i=1}^nJ_i(\bar z_n^i)&= \frac1n\sum_{i\in I_n}J_i(\bar z_n^i)+\frac1n\sum_{i\notin I_n}J_i(\bar z_n^i)\geq \frac1n\sum_{i\in I_n}J_i(\delta_i+\epsilon)+\frac1n\sum_{i\notin I_n}J_i(\delta_i)\\
&\geq\frac1n\sum_{i\in I_n}\left(J_i(\delta_i)+\tfrac12c\epsilon^2\right)+\frac1n\sum_{i\notin I_n}J_i(\delta_i)=\frac1n\sum_{i=1}^{n}J_i(\delta_i)+\frac{|I_n|}{n}\frac12c\epsilon^2,
\end{align*}
where $c>0$ is as in (A4). Sending $n\to\infty$, we obtain with help of \eqref{lim:ergodic:T1} and \eqref{claim:Lambda} the claim \eqref{pf:T1:Step1b}.

\step 2 We claim $\underline A:=\liminf_{n\to\infty} \ell_n^*\geq\mathbb{E}[\delta]$.
	
For all $\epsilon>0$, we show
\begin{equation}\label{eq:t1s2:key}
\ell_n^*\geq \frac1n\sum_{i=1}^n\delta_i-\epsilon=:k_n \qquad\forall n\in\mathbb N,
\end{equation}
which in combination with \eqref{lim:ergodic:T1} implies $\underline A:=\liminf_{n\to\infty} \ell_n^*\geq\mathbb{E}[\delta]$ by the arbitrariness of $\epsilon>0$. 

Let $\bar z_n$ be such that
$$
\frac1n\sum_{i=1}^n\bar z_n^{i}=k_n,\quad \frac1n\sum_{i=1}^nJ_i(\bar z_n^i)=M_n(k_n).
$$
We show $\bar z_n^i\leq \delta_i< z_{\rm frac}^i$ $\forall i\in\{1,\dots,n\}$, which obviously implies \eqref{eq:t1s2:key}. Indeed, the optimality condition for $\bar z_n$ implies that there exists a Lagrange-Multiplier $\Lambda\in\mathbb R$ such that $\Lambda =J_i'(\bar z_n^i)$ for all $i\in\{1,\dots,n\}$. Since $\frac1n\sum_{i=1}^n(\bar z_n^i-\delta_i)\leq -\epsilon$ there exists $\hat i\in\{1,\dots,n\}$ such that $\bar z_n^{\hat i}\in(0,\delta_i)$ and thus $J_{\hat i}'(\bar z_n^{\hat i})<0$. Hence $J_i'(\bar z_n^i)<0$ for all $i\in\{1,\dots,n\}$. Since $J'_i\geq 0$ on $(\delta_i,\infty)$ by Lemma~\ref{prop:derivatives}, we conclude that $\bar z^i_n\leq \delta_i\leq z^i_{\rm frac}$ and thus $\ell^*_n\geq k_n$ by Lemma~\ref{L:reformellstar}.
\end{proof}

\subsection{Proof of Theorem~\ref{thm:deltagleich}}
\label{sec:proofthmdeltagleich}

We begin with a preliminary structure result for minimizers of the minimum problem in the definition of $M_n(\omega,1+n^{-\frac12}D)$ for some $D>0$ see \eqref{def:Mnell}.

\begin{proposition}\label{prop:schrankeableitung}
  Let Assumption~\ref{Ass:stochasticLJ} be satisfied and assume $\delta(\omega)=1$ for $\mathbb P$-a.e.\ $\omega\in\Omega$. Fix $D>0$. There exist $\bar N\in\mathbb N$ and a sequence $(N_n)$ satisfying $N_n\to\infty$ such that the following statements hold true for $\mathbb P$-a.e.\ $\omega\in\Omega$ and $n\geq \bar N$:

  Let $\bar z_n\in\mathbb R^n$ be such that
\begin{equation}\label{eq:znprop}
\frac1n\sum_{i=1}^n\bar z_n^{i}=1+n^{-\frac12}D\quad\mbox{and}\quad \frac1n\sum_{i=1}^nJ(\tau_i\omega,\bar z_n^i)=M_n(\omega,1+n^{-\frac12}D).
\end{equation}
Then, it holds
\begin{equation}\label{schrankezn}
 \bar z_n^i\in[1,1+c^{-2}n^{-\frac12}D]\cup[N_n,\infty)\quad\text{for all}\ i\in\{1,...,n\},
\end{equation}
where $c>0$ is as in (A4).
	
\end{proposition}

\begin{proof}[Proof of Proposition~\ref{prop:schrankeableitung}]

We consider $\omega\in\Omega$ such that $\delta(\tau_i\omega)=1$ $\forall i\in\mathbb N$ and drop the dependence on $\omega$. Moreover, we use the shorthand notation $z_{\mathrm{frac}}^i:=z_{\mathrm{frac}}(\tau_i\omega)$ and $J_i(z):=J(\tau_i\omega,z)$. 

\step 1 We show that
\begin{align}\label{ersteAbleitungJ}
 0\leq J'(\bar{z}_{n}^i)\leq \frac1cDn^{-\frac{1}{2}}\quad\text{for all}\ i\in\{1,...,n\}
\end{align}
where $c>0$ is as in (A4).

\smallskip

By the optimality condition for $\bar z_n$ there exists a Lagrange-Multiplier $\Lambda\in\mathbb R$ such that $\Lambda=J_i'(\bar z_n^i)$ for all $i\in\{1,\dots,n\}$. Since $\frac1n\sum_{i=1}^n\bar z_n^i=1+n^{-\frac12}D$, there exists $i_1 \in\{1,\ldots,n\}$ such that $\bar{z}_n^{i_1}\geq 1+n^{-\frac12}D>1$. Lemma~\ref{prop:derivatives} and the assumption $\delta(\tau_i\omega)=1$ imply that $J_i$ is increasing on $(1,\infty)$ and thus we have $\Lambda\geq0$. Moreover, there exists $i_2\in\{1,\ldots,n\}$ such that $\bar{z}_n^{i_2}\leq 1+n^{-\frac12}D$. For $n$ sufficiently large such that $n^{-\frac12}D<c$, where $c>0$ as in (A4), we have (using that $J'_i(1)=0$)
\begin{align*}
0&\leq \Lambda=J'_{i_2}(\bar{z}_n^{i_2})=\int_{1}^{\bar z_n^{i_2}}J_{i_2}''(t)\,dt\stackrel{(A4)}{\leq} \frac1cn^{-\frac12}D.
\end{align*}
Since $\Lambda=J_i'(\bar z_n^i)$ for all $i\in\{1,\dots,n\}$ the claim \eqref{ersteAbleitungJ} follows.	
	
\step 2 Argument for \eqref{schrankezn}.

We firstly observe that \eqref{ersteAbleitungJ} implies $1\leq \bar z_n^i$ for all $i\in\{1,\dots,n\}$ (recall $J_i'(z)<0$ on $(0,1)$). The remaining estimates of \eqref{schrankezn} are proven in three steps.  
	
\substep{2.1} We claim that for $n$ sufficiently large, $\bar z_n^i\leq z_{\rm frac}^i$ implies $\bar z_n^i\leq 1+c^{-2}n^{-\frac12}D$, where $c>0$ is as in (A4). Indeed, using $J_i''(s)>0$ on $(0,z_{\rm frac}^i)$ and (A4), we deduce from $\bar z_n^i\leq z_{\rm frac}^i$ and $n$ sufficiently large that 
\begin{align*}
 c^{-1}Dn^{-\frac{1}{2}}\stackrel{\eqref{ersteAbleitungJ}}{\geq} J_i'\left(\bar{z}_n^i \right)=\int_1^{\bar z_n^i}J_i''(t)\,dt\stackrel{(A4)}{\geq} c \min\{\bar{z}_n^i-1,c\} .
\end{align*}
From the above inequality we deduce that $\overline z_n^i-1\geq c$ implies $n\leq D^2/c^6$. Hence, $\overline z_n^i-1<c$ and thus $1\leq \overline z_n^i\leq 1+c^{-2}Dn^{-\frac12}$ for $n>D^2/c^6$.

\substep{2.2} There exists $M<\infty$, depending only on $\psi^-(1)$ from (A2) and $c>0$ from (A4), such that 
\begin{align}\label{def:Inw}
 \sup_{n\in\mathbb N}|I_n^w|\leq M\qquad\mbox{where}\qquad I_n^w:=\left\{ i\in\{1,...,n\}\, :\, \bar{z}_n^i\geq z_{\mathrm{frac}}^i \right\}.
\end{align}
Suppose $|I_n^\omega|\geq2$ and consider some $i_n\in I_n^\omega$. Define
\begin{align}\label{competitior}
	\hat{z}_n^i:=\begin{cases}
	\bar z_n^i\quad&\mbox{if $i\notin I_n^\omega$},\\
	1&\mbox{if $i\in I_n^\omega\setminus \{i_n\}$},\\
	1+\sum_{i\in I_n^\omega}(\bar z_n^i-1)&\mbox{if $i=i_n$}.
	\end{cases}
	\end{align}
	By construction, we have $\sum_{i=1}^n\bar z_n^i=\sum_{i=1}^n\hat z_n^i$ and thus by \eqref{eq:znprop}
	\begin{align}\label{est:minifracM}
	0\geq \sum_{i=1}^n \left(J_i(\bar z_n^i)-J_i(\hat z_n^i)\right)=\sum_{i\in I_n^\omega\setminus\{i_n\}}(J_i(\bar z_n^i)-J_i(1))+J_{i_n}(\bar z_n^{i_n})-J_{i_n}(\hat z_n^{i_n}).
	\end{align}
	By the monotonicity of $J_i$ on $(1,\infty)$, (A3) and (A4), we find $\eta=\eta(c)>0$, where $c>0$ is as in (A4), such that
	\begin{equation}\label{est:minifrac}
	J_i(\bar z_n^i)-J_i(1)\geq J_i(z_{\rm frac}^i)-J_i(1)\geq \eta\qquad\forall i\in I_n^w.
	\end{equation} 
	Moreover, using $\hat z_n^{i_n}\geq1$ and thus $J_{i_n}(\hat z_n^{i_n})\leq 0$ (wich follows from the monotonicity of $J_i$ on $(1,\infty)$ and (A2)) we obtain 
	\begin{equation}\label{est:minifracM1}
	J_{i_n}(\bar z_n^i)-J_{i_n}(\hat z_n^{i_n})\geq J_{i_n}(1)\stackrel{(A2)}\geq \psi^-(1).
	\end{equation}
	Combining \eqref{est:minifracM}--\eqref{est:minifracM1}, we deduce the uniform bound $|I_n^\omega|\leq1-\eta^{-1}\psi^-(1)$.
		
	\substep{2.3} We show that there exists $(N_n)$ satisfying $N_n\to\infty$ as $n\to\infty$ such that $\bar z_n^i\geq N_n$ for all $i\in I_n^w$, where $I_n^w$ is defined in \eqref{def:Inw}.
	
	We argue by contradiction and assume that there exists $A\in[1,\infty)$ and an index $\hat i\in I_n^w$ such that $\bar z_n^{\hat i}\leq A$. For $n$ sufficiently large, we show that this contradicts \eqref{eq:znprop}. Define
	\begin{equation}
	\tilde z_n^i:=\begin{cases}1&\mbox{if $i=\hat i$},\\ \bar z_n^i+(n-|I_n^w|)^{-1}(\bar z_n^{\hat i}-1)&\mbox{if $i\notin I_n^\omega$},\\ \bar z_n^i&\mbox{if $i\in I_n^w\setminus \{\hat i\}$}.\end{cases}
	\end{equation}
        By construction, we have $\sum_{i=1}^n\tilde z_n^i=\sum_{i=1}^n\bar z_n^i$. Since $\bar z_n$ is a minimizer (see\eqref{eq:znprop}),
\begin{align*}
 0\geq\sum_{i=1}^n(J_i(\bar z_n^i)-J_i(\tilde z_n^i))=J_{\hat i}(\bar z_n^{\hat i})-J_{\hat i}(1)+\sum_{i\notin I_n^\omega}(J_i(\bar z_n^i)-J_i(\tilde z_n^i)).
\end{align*}
By \eqref{est:minifrac} we have $J_{\hat i}(\bar z_n^{\hat i})-J_{\hat i}(1)\geq\eta(c)>0$. To obtain a contradiction, it suffices to show that the second term on the right-hand side vanishes as $n$ tends to infinity. This can be seen as follows: On the one hand, we have  $\bar z_n^i\in[1,1+c^{-2}n^{-\frac12}D]$ for all $i\notin I_n^w$ by Substep~2.1, and on the other hand, we have $(n-|I_n^w|)^{-1}(\bar z_n^{\hat i}-1)\leq (n-M)^{-1}(A-1)$ thanks to $|I_n^w|\leq M$. Hence, $\bar z_n^i,\tilde z_n^i\in[1,1+\frac{c}2]$ for $n$ sufficiently large (depending only on $c$, $D$, $M$ and $A$). Now, a quadratic Taylor expansion of $J_i$ at $\bar z^i_n$ yields (using $|J''(z)|\leq c^{-1}$ for $z\in[1,1+c)$, see (A4))
\begin{align*}
  \sum_{\footnotesize\substack{i=1 \\ i\notin I_n^\omega}}^n|J_i(\bar z_n^i)-J_i(\tilde z_n^i)|\leq&\sum_{i=1}^n\left(|J_i'(\bar z_n^i)|(n-M)^{-1}(A-1)+c^{-1}(n-M)^{-2}(A-1)^2\right)\\
  \stackrel{\eqref{ersteAbleitungJ}}{\leq}&n(n-M)^{-1}c^{-1}(A-1)\Big(n^{-\frac12}D+(A-1)(n-M)^{-1}\Big)\leq Cn^{-\frac12},
\end{align*}
where $C<\infty$ depends only on $A,c,D$ and $M$.
\end{proof}

\begin{proof}[Proof of Theorem~\ref{thm:deltagleich}]

By the ergodic theorem, it holds
\begin{equation}\label{lim:ergodic:t2}
\lim_{n\to\infty}\frac1n\sum_{i=1}^n{J''(\tau_i\omega,1)^{-1}}=\mathbb E[J''(1)^{-1}],\qquad\lim_{n\to\infty}\beta_n(\omega)=\beta\end{equation}
for $\mathbb P$-a.e.\ $\omega\in\Omega$, where $\beta$ is defined in \eqref{def:alphabeta} and
\begin{equation}\label{def:beta}
\beta_n(\omega):=\min\{-J(\tau_i\omega,1)\,:\,i\in\{1,\dots,n\}\}.
\end{equation}
In Step~3 below we provide an argument for the limit $\beta_n\to\beta$.

In Step~1 and Step~2, we consider $\omega\in\Omega$ such that \eqref{lim:ergodic:t2} and the conclusion of Proposition~\ref{prop:schrankeableitung} are valid. Moreover, we drop the dependence on $\omega$ and use the shorthand notation $z_{\mathrm{frac}}^i:=z_{\mathrm{frac}}(\tau_i\omega)$ and $J_i(z):=J(\tau_i\omega,z)$. 
	
\step 1 We prove $\overline A:=\limsup_{n\to\infty}\gamma_n^*\leq\sqrt{\dfrac{\beta}{\underline{\alpha}}}$ by contradiction: Assume that  there exists $\epsilon>0$ and $\overline{N}\in\mathbb{N}$ such that 
\begin{align}\label{fN}
 \ell_n^*>1+n^{-\frac12}\sqrt{\frac{\beta}{\underline{\alpha}}}(1+\epsilon)=:k_n\quad\text{for}\  n>\overline{N}.
\end{align}
In view of Lemma~\ref{L:reformellstar} there exists $(\bar z_n)_n$ satisfying 
\begin{equation}\label{eq:Tbs1zn}
 \frac1n\sum_{i=1}^n\bar z_n^{i}=k_n,\quad \frac1n\sum_{i=1}^nJ_i(\bar z_n^i)=M_n(k_n)\quad \bar z_n^i\leq z_{\rm frac}^i\quad\forall i\in\{1,\dots,n\}.
\end{equation}
We show
\begin{align}
\limsup_{n\to\infty}n\biggl(M_n(k_n)-\frac1n\sum_{i=1}^nJ_i(1)\biggr)\leq& \beta\label{pf:T2:Step1a},\\
\liminf_{n\to\infty}n\biggl(\frac1n\sum_{i=1}^n J_i(\bar z_n^i)-\frac1n\sum_{i=1}^nJ_i(1)\biggr)\geq&\beta(1+\epsilon)^2.\label{pf:T2:Step1b}
\end{align}
Clearly, \eqref{pf:T2:Step1a} and \eqref{pf:T2:Step1b} contradict \eqref{eq:Tbs1zn} for $n$ sufficiently large.

\substep{1.1} Argument for \eqref{pf:T2:Step1b}.

We claim that there exists $K<\infty$ such that for all $n$ sufficiently large 
\begin{align}\label{melgl0}
 n\biggl(\frac1n\sum_{i=1}^nJ_i(\bar z_n^i)-\frac1n\sum_{i=1}^nJ_i(1)\biggr)&\geq \biggl(\frac1{n}\sum_{i=1}^n\big(\tfrac12J_i''(1)\big)^{-1}\biggr)^{-1}\frac{\beta}{\underline \alpha}(1+\epsilon)^2-\frac{K}{\sqrt{n}},
\end{align}
where $\bar \alpha$ and $\beta$ are defined in \eqref{def:alphabeta}. Note that \eqref{lim:ergodic:t2} and \eqref{melgl0} imply \eqref{pf:T2:Step1b}.

We prove \eqref{melgl0}. By \eqref{fN}, \eqref{eq:Tbs1zn}, and Proposition~\ref{prop:schrankeableitung} (applied with $D=\sqrt{\frac{\beta}{\underline\alpha}}(1+\epsilon)^2$), we get
\begin{equation}\label{neu:6}
  1\leq z_n^i\leq 1+n^{-\frac12}C
\end{equation}
for some $C<\infty$ independent of $n$. Hence, a Taylor expansion yields
\begin{align}\label{melgl}
\sum_{i=1}^{n}J_i(\bar z_n^i)=\sum_{i=1}^nJ_i(1)+\frac1{2}\sum_{i=1}^{n}J_i''(1)\left({\bar z}_n^i-1\right)^2+\frac1{6}\sum_{i=1}^{n}J_i'''(\xi_n^i)\left(\bar z_n^i-1\right)^3,
\end{align}
where $\xi_n^i\in\left[1,\bar z_n^i\right]$.
To estimate the second term on the right-hand side, note that Cauchy-Schwarz' inequality yields 
\begin{equation*}
\Big(\sum_{i=1}^n(\bar z_n^i-1)\Big)^2\leq\Big(\frac12\sum_{i=1}^nJ_i''(1)(\bar z_n^i-1)^2\Big)\Big(\sum_{i=1}^n(\tfrac12J_i''(1))^{-1}\Big).
\end{equation*}
Combined with the identity  $\sum_{i=1}^{n}({\bar z}_n^i-1)=n(k_n-1)=\sqrt{n}\sqrt{\frac{\beta}{\underline \alpha}}(1+\epsilon)$  we get
\begin{align}\label{1dmin}
\biggl(\frac1{n}\sum_{i=1}^n\big(\tfrac12J_i''(1)\big)^{-1}\biggr)^{-1}\frac{\beta}{\underline \alpha}(1+\epsilon)^2\leq \frac1{2}\sum_{i=1}^{n}J_i''(1)\left(\bar z_n^i-1\right)^2.
\end{align}

Moreover, \eqref{neu:6} and (A4) imply for $n$ sufficiently large
\begin{equation}\label{eq:t2:s1:final}
\frac1{6}\sum_{i=1}^{n}J_i'''(\xi_n^i)\left(\bar z_n^i-1\right)^3\geq-\frac{C^3}{6c\sqrt{n}}.
\end{equation}
Clearly, \eqref{melgl}--\eqref{eq:t2:s1:final} imply \eqref{melgl0} (with $K=\frac{C^3}{6c}$).

\substep{1.2} Argument for \eqref{pf:T2:Step1a}.

For every $n\in\mathbb N$, we choose $\hat i_n\in\{1,\dots,n\}$ such that $-J_{\hat i_n}(1)=\beta_n$ (see \eqref{def:beta}) and define $z_n\in\mathbb R^n$ as 
\begin{equation*}
z_n^i=\begin{cases}1&\mbox{if }i\in\{1,\dots,n\}\setminus \{\hat i_n\}\\1+n(k_n-1)&\mbox{if }i=\hat i_n\end{cases}.
\end{equation*}
Since $\frac1n\sum_{i=1}^nz_n^i=k_n=1+n^{-\frac12}\sqrt{\frac{\beta}{\underline{\alpha}}}(1+\epsilon)$, we have
\begin{align*}
n\biggl(M_n(k_n)-\frac1n\sum_{i=1}^nJ_i(1)\biggr)\leq& J_{\hat i_n}(1+n(k_n-1))-J_{\hat i_n}(1)\\
\leq& \psi^+\biggl(1+\sqrt{n}\sqrt{\frac{\beta}{\underline\alpha}}(1+\epsilon)\biggr)+\beta_n,
\end{align*}
where the second inequality holds by (A2) and the choice of $\hat i_n$. Now, \eqref{pf:T2:Step1a} follows from \eqref{lim:ergodic:t2} and assumption (A2).

\step 2 Proof of $\underline A:=\liminf_{n\to\infty}\gamma_n^*\geq\sqrt{\tfrac{\beta}{\underline{\alpha}}}$.
	
We show that for every $\epsilon>0$ there exists $\bar N\in\mathbb N$ such that
\begin{align}\label{fN2}
 \ell_n^*\geq1+n^{-\frac12}\sqrt{\frac{\beta}{\underline{\alpha}}}(1-\epsilon)=:k_n\quad\text{for}\  n>\overline{N}.
\end{align}
Note that \eqref{fN2} implies $\liminf_{n\to\infty}\gamma_n^*\geq \sqrt{\frac{\beta}{\underline{\alpha}}}(1-\epsilon)$ for all $\epsilon>0$, and thus the claim.

Let $(\bar z_n)_n$ be a sequence satisfying for all $n\in\mathbb N$,
\begin{equation}\label{eq:barznkn}
\frac1n\sum_{i=1}^n\bar z_n^i=k_n,\qquad M_n(k_n)=\frac1n\sum_{i=1}^nJ_i(\bar z_n^i).
\end{equation}
To prove \eqref{fN2}, we only need to show that
\begin{equation}\label{neu:10}
  z_n^i\leq z_{\rm frac}^i\text{ for all $i\in\{1,\dots,n\}$ for $n$ sufficiently large},
\end{equation}
depending only on $\underline\alpha$ $\beta$, $c$, and $\epsilon>0$.

\substep{2.1} We show that
\begin{equation}\label{eq:t2s22}
 \limsup_{n\to\infty}n\biggl(M_n(k_n)-\frac1n\sum_{i=1}^nJ_i(1)\biggr)\leq \beta(1-\epsilon).
\end{equation}
Set
$$
\hat{z}_n^i:=1+n^{-\frac12}\sqrt{\frac{\beta}{\underline{\alpha}}}(1-\epsilon)\left(\frac1n\sum_{i=1}^{n}\frac{1}{\alpha_i}\right)^{-1}\frac{1}{\alpha_i},
$$
where $\alpha_i:=\frac12 J_i''(1)$. By construction, we have 
\begin{equation}\label{eq:hatznt2s2}
\frac1n\sum_{i=1}^n\hat z_n^i=k_n,\qquad 0\leq \hat z_n^i-1\leq n^{-\frac12}C
\end{equation}
where $C<\infty$ depends only on $\underline \alpha,\beta$ and $c>0$ from (A4) (note that (A4) implies $\alpha_i\leq \frac1{2c}$ and $\frac1{\alpha_i}\leq \frac2c$). Hence, a Taylor expansion of $J_i$ at $1$ and (A4) yield for $n$ sufficiently large 
\begin{align*}
 \sum_{i=1}^n \left(J_i(\hat z_n^i)-J_i(1)\right)&\leq\sum_{i=1}^{n}\alpha_i\left(\hat{z}_n^i-1\right)^2+\frac1{6c}\sum_{i=1}^{n}\left(\hat{z}_n^i-1\right)^3\\
	&\leq\frac{\beta}{\underline{\alpha}}(1-\epsilon)^2\biggl(\frac1n\sum_{i=1}^n\frac1{\alpha_i}\biggr)^{-1}+\frac{C^3}{6c}n^{-\frac12},	\end{align*}
	where $C<\infty$ is the same as in \eqref{eq:hatznt2s2}. Finally, \eqref{lim:ergodic:t2} implies $(\frac1n\sum_{i=1}^n\frac1{\alpha_i})^{-1}\leq \underline \alpha(1+\epsilon)$ for $n$ sufficiently large and thus \eqref{eq:t2s22} follows.
	
        \substep{2.2} We now prove \eqref{neu:10} by contraposition. Suppose $\bar z_n^{\hat i}>z_{\rm frac}^{\hat i}$ for some $\hat i\in\{1,\dots,n\}$. Then Proposition~\ref{prop:schrankeableitung} yields $\bar z_n^{\hat i}\geq N_n$ for some $(N_n)$ with $N_n\to \infty$, and thus $J_{\hat i}(\bar z_n^{\hat i})\geq -\sup_{s\geq N_n}\psi^+(s)$ by (A2). Hence, with $J_i(\bar z_n^i)\geq J_i(1)$ and $-J_{\hat i}(1)\geq\beta$, we therefore get
\begin{align*}
  \sum_{i=1}^n(J_i(\bar z_n^i)-J_i(1))\geq J_{\hat i}(\bar z_n^{\hat i})-J_{\hat i}(1)\geq \beta-\sup_{s\geq N_n}\psi^+(s).
\end{align*}
Since $\sup_{s\geq N_n}\psi^+(s)\to 0$ for $n\to\infty$, the above lower bound combined with the upper bound \eqref{eq:t2s22} and \eqref{eq:barznkn} yields a contradiction for $n$ sufficiently large, and thus \eqref{neu:10} follows.

\step 3 Argument for $\beta_n\to \beta$ almost surely in \eqref{lim:ergodic:t2}.

The sequence $(\beta_n(\omega))_n\subset\mathbb R$ is decreasing and it holds $\beta_n(\omega)\geq \beta$ for all $n\in\mathbb N$. Hence, there exists $\hat \beta(\omega)\geq \beta$ such that
$$
\lim_{n\to\infty}\beta_n(\omega)=\hat \beta(\omega)\geq \beta.
$$
It remains to show that $\hat\beta(\omega)=\beta$ for $\mathbb P$-a.e.\ $\omega\in\Omega$. We argue by contradiction and therefore suppose that there exist $\epsilon>0$ and a set $\Omega'\subset\Omega$ with positive measure such that $\hat\beta(\omega)\geq \beta+\epsilon$ for all $\omega\in \Omega'$. Then we obtain for all $\omega\in\Omega'$ that
$$
\limsup_{n\to\infty}\frac1n\sum_{i=1}^n\chi_{\left\{-J(\tau_i\omega,1)\leq \beta+\frac12\epsilon\right\}}(\tau_i\omega)=0,
$$
where $\chi_A$ denotes the indicator function. Clearly this contradicts the ergodic theorem and the definition of $\beta$ in the form
\begin{equation*}
\lim_{n\to\infty}\frac1n\sum_{i=1}^n\chi_{\left\{-J(\tau_i\omega,1)\leq \beta+\frac12\epsilon\right\}}=\mathbb E\left[\chi_{\left\{-J(1)\leq \beta+\frac12\epsilon\right\}}\right]>0\qquad\mbox{for $\mathbb P$-a.e.\ $\omega\in \Omega$}.
\end{equation*}
Hence the theorem is proven.
\end{proof}

\textbf{Acknowledgments.} During the work on this project, LL was affiliated most of the time with the Institute of Mathematics at the University of Würzburg, Germany.  LL gratefully acknowledges the kind hospitality of the Technische Universität Dresden during her research visits, which were partially funded by the Deutsche Forschungsgemeinschaft
(DFG, German Research Foundation) – within project 405009441 and TU Dresden’s
Institutional Strategy “The Synergetic University”.


\begin{thebibliography}{9}

\bibitem{AlicandroCicaleseGloria2011} R.\ Alicandro, M.\ Cicalese and A.\ Gloria, Integral representation results for energies defined on stochastic lattices and application to nonlinear elasticity, \textit{ Arch.\ Ration.\ Mech.\ Anal.} \textbf{200} (2011), 881--943. 
\bibitem{BraidesDalMasoGarroni1999} A.\ Braides, G.\ Dal Maso and A.\ Garroni, Variational formulation of softening phenomena in fracture mechanics: The one-dimensional case, \textit{Arch.\ Ration.\ Mech.\ Anal.} \textbf{146} (1999), 23--58.\ 
 \bibitem{BraidesGelli2002} A.\ Braides and M.\ S.\ Gelli, Continuum limits of discrete systems without convexity hypotheses, \textit{Math.\ Mech.\ Solids}, {\bf 7} (2002), 41--66.
 \bibitem{BraidesLewOrtiz2006} A.\ Braides, A.\ Lew and M.\ Ortiz, Effective cohesive behavior of Layers of interatomic planes, \textit{Arch.\ Ration.\ Mech.\ Anal.} \textbf{180} (2006), 151--182.
 \bibitem{BraidesSolci2016} A.~Braides and M.~Solci, Asymptotic analysis of Lennard-Jones systems beyond the nearest-neighbour setting: A one-dimensional prototypical case, \textit{Math.\ Mech.\ Solids} \textbf{21} (2016), 915--930.

\bibitem{DalMasoModica1985} G. Dal Maso and L. Modica, Nonlinear stochastic homogenization and ergodic theory, \textit{J.\ Reine Angew.\ Math.} {\bf 368} (1986), 28--42.
\bibitem{FriedrichStefanelli2020}
M.~Friedrich and U.~Stefanelli, Crystallization in a One-Dimensional Periodic Landscape. {\em J. Stat. Phys.} \textbf{179} (2020), 485--501.
\bibitem{HallHudsonvanMeurs2018}
C.L.~Hall, T.~Hudson and P.~van Meurs, Asymptotic Analysis of Boundary Layers in a Repulsive Particle System. {\em Acta Appl. Math.} \textbf{153} (2018), 1--54.
 \bibitem{IosifescuLichtMichaille2001} O.\ Iosifescu, C.\ Licht and G.\ Michaille, Variational limit of a one dimensional discrete and statistically homogeneous system of material points, \textit{Asymptot.\ Anal.} \textbf{28} (2001), 309--329.
\bibitem{JansenKoenigSchmidtTheil2021}
S.~Jansen, W.~König, B.~Schmidt and F.~Theil, Surface Energy and Boundary Layers for a Chain of Atoms at Low Temperature. {\em Arch. Rational Mech. Anal.} \textbf{239} (2021), 915--980.
\bibitem{KimuravanMeurs2021}
M.~Kimura and P.~van Meurs, Quantitative estimate of the continuum approximations of interacting particle systems in one dimension. {\em 
SIAM J. Math. Anal.} \textbf{53} (2021),  681--709.
\bibitem{Lauerbach-Diss}
L.~Lauerbach, \textit{Stochastic Homogenization in the Passage from Discrete to Continuous Systems---Fracture in Composite Materials}, Ph.D.~thesis, University of Würzburg, 2020.
\bibitem{unserpaper1} L.~Lauerbach, N.~Neukamm, M.~Schäffner and A.~Schlömerkemper, Mechanical behaviour of heterogeneous nanochains in the $\Gamma$-limit of stochastic particle systems. arXiv:1909.06607.
\bibitem{unserpaper2} L.~Lauerbach and A.~Schlömerkemper, Derivation of a variational model for brittle fracture from a random heterogeneous particle chain. arXiv:2104.08607. 
\bibitem{LauerbachSchaeffnerSchloemerkemper2017} L.~Lauerbach, M.~Schäffner and A.~Schlömerkemper, On continuum limits of heterogeneous discrete systems modelling cracks in composite materials, \textit{GAMM‐Mitt.} \textbf{40} (2017), 184--206.
\bibitem{NeukammSchaeffnerSchloemerkemper2017} S.\ Neukamm, M.\ Sch\"affner and A.\ Schl\"omerkemper, Stochastic homogenization of nonconvex discrete energies with degenerate growth, \textit{SIAM J.\ Math.\ Anal.} \textbf{49} (2017), 1761--1809.
\bibitem{ScardiaSchloemerkemperZanini2011} L.\ Scardia, A.\ Schl\"omerkemper and C.\ Zanini, Boundary layer energies for nonconvex discrete systems, \textit{Math.\ Models Methods Appl.\ Sci.} \textbf{21} (2011), 777--817.
 \bibitem{SchaeffnerSchloemerkemper2015} M.\ Sch\"affner and A.\ Schl\"omerkemper, On Lennard-Jones systems with finite range interactions and their asymptotic analysis, \textit{Netw. Heterog. Media} \textbf{13} (2018), 95--118.
 \bibitem{Truskinovsky1996} L.\ Truskinovsky, Fracture as a phase transition, in \textit{Contemporary Research in the Mechanics and Mathematics of Materials} (1996), 322--332.

	
\end{thebibliography}
\end{document}